\documentclass{amsart}
\usepackage{amsmath,amscd,amsthm,amsfonts,a4wide,mathrsfs}

\theoremstyle{plain}
\newtheorem{theorem}                 {Theorem}      [section]
\newtheorem{proposition}  [theorem]  {Proposition}
\newtheorem{corollary}    [theorem]  {Corollary}

\theoremstyle{definition}
\newtheorem{example}      [theorem]  {Example}
\newtheorem{remark}       [theorem]  {Remark}

\def \theo-intro#1#2 {\vskip .25cm\noindent{\bf Theorem #1\ }{\it #2}}

\numberwithin{equation}{section}

\def \g{\mathfrak{g}}
\def \k{\mathfrak{k}}
\def \h{\mathfrak{h}}
\def \m{\mathfrak{m}}
\def \p{\mathfrak{p}}

\def \L{\mathscr{L}} 
\def \d{\mathrm{d}}

\def \ip#1#2{\langle#1,#2\rangle}

\DeclareMathOperator{\volume}{\ast 1}

\DeclareMathOperator{\musicd}{\sharp}
\DeclareMathOperator{\musicu}{\flat}
\DeclareMathOperator{\intpr}{\lrcorner}
\DeclareMathOperator{\dive}{div}

\def \rn{\mathbb{R}}
\def \cn{\mathbb{C}}
\def \V{\mathcal{V}}
\def \H{\mathcal{H}}

\newcommand{\R}{{\mathbb{R}}}

\newcommand{\C}{{\mathbb{C}}}

\newcommand{\beq}{\begin{equation}} 
\newcommand{\eeq}{\end{equation}} 
\newcommand{\bea}{\begin{eqnarray}} 
\newcommand{\eea}{\end{eqnarray}} 
\newcommand{\ben}{\begin{eqnarray*}} 
\newcommand{\een}{\end{eqnarray*}} 
\newcommand{\ra}{\rightarrow}

\allowdisplaybreaks

\begin{document}\larger[2]\setlength{\baselineskip}{1.0\baselineskip}
\begin{footnotesize}
\begin{flushright}
CP3-ORIGINS-2010-10
\end{flushright}
\end{footnotesize}
\bigskip

\title{Some global minimizers of a symplectic Dirichlet energy}
\author{J.M.~Speight}
\address{School of Mathematics, University of Leeds, Leeds, LS2 9JT} 
\email{speight@maths.leeds.ac.uk }
\author{M.~Svensson}
\address{Department of Mathematics and Computer Science, and $\mathrm{CP}^3$-Origins, Centre of Excellence for Particle Physics Phenomenology, University of Southern Denmark, Campusvej 55, DK-5230 Odense M}
\email{svensson@imada.sdu.dk}

\begin{abstract} 
The variational problem for the functional 
$F=\frac12\int_M\|\varphi^*\omega\|^2$ is considered,
where $\varphi:(M,g)\ra (N,\omega)$ maps a 
Riemannian manifold
to a symplectic manifold. This functional arises
in theoretical physics as the strong coupling limit of the Faddeev-Hopf
energy, and may be regarded as a symplectic analogue of the Dirichlet
energy familiar from harmonic map theory.
The Hopf fibration $\pi:S^3\ra S^2$ is known to be a locally stable critical
point of $F$. It is proved here that $\pi$ in fact minimizes $F$ in its
homotopy class and this result is extended to the case where $S^3$ is 
given the metric of the Berger's sphere.
It is proved that if $\varphi^*\omega$ is coclosed then 
$\varphi$ is a critical point of $F$ and minimizes $F$ in its homotopy
class. If $M$ is a compact Riemann surface, it is proved that every
critical point of $F$ has $\varphi^*\omega $ coclosed. A family of
holomorphic homogeneous projections into Hermitian symmetric spaces is
constructed and it is proved that these too minimize $F$ in their
homotopy class. 
\end{abstract}

\maketitle

\section{Introduction}

Given a smooth map $\phi:M\ra N$ between Riemannian manifolds, a natural
notion of the energy of $\phi$ is the Dirichlet energy,
$$
E=\frac12\int_M\|\d\phi\|^2.
$$
The variational problem for $E$, whose critical points are
harmonic maps,  has been heavily studied for many years. If we replace the
metric on $N$ by a symplectic form $\omega$, a natural analogue of 
$E$ is
$$
F=\frac12\int_M\|\phi^*\omega\|^2,
$$
which one can regard as a kind of symplectic Dirichlet energy. In this 
paper we study the variational problem for $F$, focussing particularly
on the problem of obtaining global minimizers of $F$.

Our original motivation comes from theoretical physics. If 
$N$ is a K\"ahler manifold, both $E$ and $F$ are defined, and the 
variational problem for $E+\alpha F$, where $\alpha$ is a positive constant 
called the coupling constant, is known to physicists as the
Faddeev-Hopf (or Faddeev-Skyrme) model. The case of most interest is
$M=\R^3$, $N=S^2$. This variational problem, originally proposed by
Faddeev in the 1970s, lay dormant for lack of computational
power until the 1990s, but has been the subject of
intense numerical study in recent years
\cite{batsut,fadnie,hiesal,sut2}. Its critical points, which are
interpreted as topological solitons, have been proposed as models of
gluon flux tubes in hadrons (particles composed of quarks). 
There has been some analytic progress on this model too. Kapitanski and
Vakulenko proved a rather curious topological energy bound for 
maps $\R^3\ra S^2$ \cite{kapvak},
Kapitanski and Auckly proved weak existence (i.e.\ existence
in some Sobolev space with low regularity) of global minimizers in every
homotopy class of maps $M\ra S^2$, with $M$ a compact oriented 3-manifold
\cite{auckap},
and Ward obtained some exact results for the case $S^3\ra S^2$ \cite{war}.

The variational problem for $F$, which we shall consider here,
 can be interpreted as
the Faddeev-Hopf model in the strong coupling limit $\alpha\ra\infty$.
This has two key similarities with Yang-Mills theory: it is conformally
invariant in dimension 4 and it possesses an infinite dimensional
symmetry group (the symplectic diffeomorphisms of $N$). 
As in Yang-Mills, the most physically relevant choice of $M$ is $S^4$,
regarded as the conformal compactification of Euclideanized
spacetime $\R^4$. It is an interesting open question
whether there is a critical point of $F$ in the generator of
$\pi_4(S^2)$. Such a critical point would be interpreted as an
instanton in the strongly coupled Faddeev-Hopf model on Minkowski space.

This variational problem seems to have received remarkably little
attention. Ferreira and De Carli \cite{decfer} analyzed the case
$M=S^3\times\R$, with a Lorentzian metric, and $N=\C$, $S^2$ or the 
hyperbolic plane, working within a particular rotationally invariant
ansatz. The first systematic development of the variational calculus
was made in \cite{Speight-Svensson}.  We begin by briefly reviewing
some results from that paper. Throughout, all maps are smooth,
$(M,g)$ is a Riemannian manifold and $(N,\omega)$ is a symplectic manifold.
In the case where
$N$ is K\"ahler, $\omega$ is the K\"ahler form $\omega(\cdot,\cdot)=
h(J\cdot,\cdot)$, $h$ is the metric and $J$ is the almost complex structure.
The coderivative on differential forms will be denoted $\delta$.

\begin{theorem}\cite{Speight-Svensson}\label{first}  For any
variation $\varphi_t$ of $\varphi:M\ra N$ 
with variational vector field $X\in\Gamma(\varphi^{-1}TN)$ we have
$$\frac{d}{dt}F(\varphi_t)\big\vert_{t=0}=\int_M\omega(X,\d\varphi(\musicd\delta\varphi^*\omega))\volume.$$  
In particular, $\varphi$ is a critical point of $F$ if and only if
$$\d\varphi(\musicd\delta\varphi^*\omega)=0.$$  
\end{theorem}

\begin{theorem}\cite{Speight-Svensson} A Riemannian submersion 
$\varphi:M\ra N$ from a Riemannian manifold to a K\"ahler manifold 
is a
critical point of $F$ if 
and only if it has minimal fibres, i.e., if and only if it is a
harmonic map. 
\end{theorem}

A well known harmonic Riemannian submersion is the Hopf
map $$\pi:S^3\to S^2,$$ where $S^3$ is the sphere in $\rn^4$ of
radius $1$, and $S^2$ the sphere in $\rn^3$ of
radius $1/2$. As a harmonic map, $\pi$ is well known to be
\emph{unstable}, as are all harmonic maps from $S^3$ \cite{Xin}. We
studied in \cite{Speight-Svensson} the second variation of $F$ for
maps into K\"ahler manifolds, and found the associated Jacobi
operator.
 
\begin{theorem}\cite{Speight-Svensson} Let $N$ be K\"ahler and
$\varphi:M\to N$ be a critical point of
$F$. Then the Hessian of $\varphi$ is
$$H_\varphi(X,Y)=\int_Mh(X,\L_\varphi Y)\volume,$$ where
$$\L_\varphi Y=-J\{\nabla^{\varphi}_{Z_\varphi}Y+\d\varphi[\musicd\delta\d\,
\varphi^*(Y\intpr\omega)]\}\quad\text{and}\quad
Z_\varphi=\musicd\delta\varphi^*\omega.$$
\end{theorem}
By a careful calculation of the spectrum of this operator for the Hopf
map, we proved the following result, which was
conjectured by Ward \cite{war}.
\begin{theorem}\label{as}\cite{Speight-Svensson} The
Hopf map $\pi:S^3\to S^2$ is stable for the full Faddeev-Hopf
functional $E+\alpha F$  if and only if $\alpha\geq 1$. 
\end{theorem}
In particular, the Hopf map is a stable critical point of $F$.
In this paper, we strengthen
this result to show that the Hopf map in fact \emph{minimizes}
$F$ in its homotopy class. The same is true for the Hopf map from the
Berger's spheres $$\pi:(S^3,g_t)\to S^2,$$ for all $0<t\leq1$; see Example \ref{ex:berger} for the definition of Berger's spheres. It is interesting
to note that a slightly stronger version of Theorem \ref{as} (namely, that
$\pi$ is a {\em local} minimizer of $E+\alpha F$ when $\alpha>1$) was obtained 
independently by Isobe \cite{iso} using rather different methods. It
remains an open question whether the Hopf map {\em globally} 
minimizes $E+\alpha F$
for some $\alpha$. 

As proved in \cite{Speight-Svensson}, we have 
$$H_\varphi(X,X)=\int_M\omega(X,\nabla^\varphi_{Z_\varphi}X)\volume
+\|\d\varphi^*(X\intpr\omega)\|^2_{L^2}\qquad(X\in\Gamma(\varphi^{-1}TN).$$  
In particular, when $\varphi^*\omega$ is coclosed (so $Z_\varphi=0$), 
$\varphi$ is
a stable critical point of $F$. 
In this paper, we strengthen this by showing that, if
$\varphi^*\omega$ is coclosed, then $\varphi$ actually minimizes $F$
in its homotopy class.

\section{Global minimizers}

Denote by $S^3$ the unit sphere in $\rn^4$ of radius 
$1$, and by $S^2$ the sphere in $\rn^3$ of  radius $1/2$
.
\begin{theorem} The Hopf map $\pi:S^3\to S^2$ minimizes $F$ in its
homotopy class.
\end{theorem}

The proof makes use of the \emph{Hopf invariant} of a map
$\varphi:S^3\to S^2$.
Recall that this is defined as the number 
$$H(\varphi)=\int_{S^3}\d\alpha\wedge\alpha,$$ where $\d\alpha=\varphi^*\omega$, and
$\omega$ is the volume form of $S^2$. It is well known that this is
 independent of the choice of $\alpha$, and 
depends only on the homotopy class of $\varphi$, see e.g. \cite{Bott-Tu}.

\begin{proof} As above, assume that $\varphi:S^3\to S^2$ is any map and
write $\varphi^*\omega=\d\alpha$. By the Hodge decomposition, we may assume 
that
$\alpha$ is coexact. Then
$$F(\varphi)=\frac{1}{2}\ip{\d\alpha}{\d\alpha}_{L^2}=
\frac{1}{2}\ip{\alpha}{\Delta\alpha}_{L^2}\geq\frac{\lambda_1}{2}\|\alpha\|_{L^2}^2,$$
where $\lambda_1$ is the first  eigenvalue of the Hodge-Laplace
operator on coexact $1$-forms on $S^3$. It is known that $\lambda_1=4$, see
e.g., \cite[page 270]{Gallot-Meyer}. Hence
$$\|\alpha\|^2_{L^2}\leq\frac{1}{2}F(\varphi).$$ By Cauchy-Schwarz,
$$H(\varphi)\leq\|\varphi^*\omega\|_{L^2}\|\alpha\|_{L^2}\leq\frac{1}{\sqrt{2}}\|\varphi^*\omega\|_{L^2}\sqrt{F(\varphi)}=F(\varphi).$$   

For the Hopf map $\pi:S^3\to S^2$, it is well known that (see
e.g.\ \cite[page 102]{Gilkey})
$$\d\ast\pi^*\omega=2\pi^*\omega,$$ so that
$$H(\pi)=F(\pi).$$ Thus, if $\varphi$ and $\pi$ are
homotopic, then $$F(\pi)=H(\pi)=H(\varphi)\leq F(\varphi).$$
\end{proof}

\begin{example}\label{ex:berger} Consider again the Hopf map $\pi:S^3\to S^2$. We
may write the metric $g$ on $S^3$ as $$g=g_\V+g_\H,$$ where $\V$ is
the distribution tangent to the fibres of $\pi$, and $\H$ its
orthogonal complement. For $0<t<1$, the 3-dimensional \emph{Berger's sphere} is the
Riemannian manifold $(S^3,g_t)$, where $$g_t=t^2g_\V+g_\H.$$ 

It is easy to see that $$\d\ast\pi^*\omega=2t\pi^*\omega.$$
For $0<t<1$, a simple calcuation shows that the sectional curvature of
$g_t$ is bounded below by $t^2$. Hence, the minimal eigenvalue
for the Hodge-Laplace operator on coexact 1-forms with respect to
$g_t$ is bounded below by $4t^2$, see \cite[page 270]{Gallot-Meyer}. 

Thus, for any map $\varphi$ homotopic to $\pi$, a calculation similar
to that above gives $$F(\pi)=tH(\pi)=tH(\varphi)\leq F(\varphi).$$
So $\pi$ still minimizes $F$ in its homotopy class, as a map
from the Berger's sphere to $S^2$. 
\end{example}

Next we consider another class of maps which minimize $F$.

\begin{theorem}\label{theorem:coclosed}\label{phi}
Let $M$ be compact and $\varphi:M\ra N$ have $\varphi^*\omega$ coclosed.
 Then $\varphi$ minimizes $F$ in its homotopy class. 
\end{theorem}

\begin{proof} 
Let $\varphi:M\ra N$ have $\varphi^*\omega$ coclosed, $\psi:M\ra N$
be homotopic to $\varphi$ and $\varphi_t$ be a smooth homotopy
from $\varphi$ to $\psi$. By the homotopy lemma \cite{Eells-Lemaire}
$$\frac{d}{dt}\ip{\varphi_t^*\omega}{\varphi^*\omega}_{L^2}=
\ip{\d(\varphi_t^*X\intpr\omega)}{\varphi^*\omega}_{L^2}=
\ip{\varphi_t^*X\intpr\omega}{\delta\varphi^*\omega}_{L^2}=0.$$  
Hence, by Cauchy-Schwartz,
$$F(\varphi)=\frac{1}{2}\ip{\psi^*\omega}{\varphi^*\omega}_{L^2}\leq\sqrt{F(\psi)F(\varphi)},$$
so that $$F(\varphi)\leq F(\psi).$$
\end{proof}

\begin{remark} Note that if $\delta\varphi^*\omega=0$ then
$\varphi^*\omega$ is harmonic, and thus minimizes the $L^2$-norm in its
cohomology class. If $\psi$ is homotopic to $\varphi$, then
$\psi^*\omega$ is in this cohomology class, and this gives an alternative
proof of the above result. 
\end{remark}

\begin{corollary} Any critical immersion from a compact Riemannian manifold 
to a symplectic manifold minimizes $F$ in its homotopy class.
\end{corollary}

Clearly, any symplectomorphism on a compact symplectic Riemannian manifold 
is a
minimizer in its homotopy class. In particular, we have
\begin{corollary} The identity map on a compact symplectic Riemannian 
manifold
minimizes $F$ in its homotopy class.
\end{corollary}

We next prove that if $M$ is a compact Riemann surface, all critical
points of $F$ have $\varphi^*\omega$ coclosed.

\begin{theorem} Let $M$ be a compact, oriented, 2-dimensional Riemannian
manifold and $N$ be a symplectic manifold. Then every critical point
$\varphi:M\ra N$ of $F$ has $\varphi^*\omega$ coclosed (and hence
minimizes $F$ in its homotopy class).
\end{theorem}

\begin{proof}
For any $\varphi:M\ra N$,
 $$\varphi^*\omega=f\volume,$$ where $\volume$ is the volume form on $M$ 
and $f$ is a real function on
$M$. Then $\delta\varphi^*\omega=\ast\d f$, so that
$$\musicd\delta\varphi^*\omega=J\nabla f,$$ where $J$ is the Hermitian
structure on $M$ associated with the orientation.  

Assume now that $\varphi$ is critical. Then $\d\varphi(J\nabla f)=0$
so that $$0=(\varphi^*\omega)(J\nabla f,\nabla f)=-f|\nabla f|^2.$$ 
It follows that $\nabla f=0$ everywhere, so $f$ is constant. 
(Assume, to the contrary, that $(\nabla f)(p)\neq 0$ for some $p\in M$. 
Then there
is some neighbourhood of $p$ on which $\nabla f\neq 0$. But then
$f=0$ on this neighbourhood, so $(\nabla f)(p)=0$, a contradiction.)
Hence $\varphi^*\omega$ is coclosed.  
\end{proof} 

For our next result, recall that a submersion $\varphi:(M,g)\to
(N,h,J)$ from a Riemannian manifold to an almost Hermitian manifold
gives rise to an $f$-structure $f$ on $M$, such that $\ker
f=\ker\d\varphi$ and the restriction of $f$ to $(\ker\d\varphi)^\perp$
corresponds to $J$ under the identification
$(\ker\d\varphi)^\perp\cong\varphi^*TN$ (see \cite{Rawnsley} for the definition of an $f$-structure). The map is said to be
\emph{pseudo horizontally (weakly) conformal} (\emph{PHWC}) if $f$ is
skew-symmetric with respect to the metric $g$ on $M$. In particular,
if $M$ also carries an almost Hermitian structure with respect to
which $\varphi$ is holomorphic, then $\varphi$ is \emph{PHWC}. See
e.g. \cite{Baird-Wood,Radu} for other characterizations of \emph{PHWC}
maps.  
\begin{proposition} Assume that $\varphi:(M,g)\to (N,h,J)$ is a PHWC
submersion from a Riemannian manifold to an almost Hermitian manifold,
with associated $f$-structure $f$. Then 
$$\delta\varphi^*\omega=f\dive
f\intpr\varphi^*\omega-\sum_ah(\varphi_*fe_a,\nabla\d\varphi(e_a,\
\cdot\ )),$$ where in the last term we sum over a local orthonormal frame for
$TM$. 
\end{proposition}
\begin{proof} For any vector field $Y$ on $M$ we have 
\begin{equation*}
\begin{split}
\delta\varphi^*\omega(Y)=&
-\sum_a\big(e_a\varphi^*\omega(e_a,Y)-\varphi^*\omega(\nabla_{e_a}e_a,Y)
-\varphi^*\omega(e_a,\nabla_{e_a}Y)\big)\\ 
=&-\sum_a\big(h(\nabla_{e_a}\varphi_*fe_a-\varphi_*f\nabla_{e_a}e_a,\varphi_*Y)
+h(\varphi_*fe_a,\nabla\d\varphi(e_a,Y))\big)\\ 
=&-\sum_ah(\nabla\d\varphi(e_a,fe_a),\varphi_*Y)-h(\varphi_*\dive f,\varphi_*Y)\\
&-\sum_ah(\varphi_*fe_a,\nabla\d\varphi(e_a,Y))
\end{split}
\end{equation*} 
To see that the first term is zero, we can choose the local
orthonormal frame such that $e_1,\dots,e_n$ span $\ker\d\varphi=\ker
f$, while $e_{n+1},\dots,e_m$ span the orthogonal complement of
$\ker\d\varphi$. Evaluating the sum for the two local orthonormal
frames $\{e_1,\dots,e_n,e_{n+1},\dots,e_m\}$ and
$\{e_1,\dots,e_n,fe_{n+1},\dots,fe_m\}$, shows that this amounts to
zero. Furthermore, the second term equals $f\dive
f\intpr\varphi^*\omega(Y)$, and the proof is finished.  
\end{proof}
Assume that $\varphi:(M,g)\to(N,h)$ is a submersion between two
Riemannian manifolds, and denote by $\V$ the \emph{vertical}
distribution $\ker\d\varphi$ and by $\H$ the \emph{horizontal}
distribution $\V^\perp$ on $M$. Recall that $\varphi$ is said to be
\emph{horizontally conformal} if there is a positive function
$\lambda$ on $M$, the dilation of $\varphi$, such that
$$h(\varphi_*X,\varphi_*Y)=\lambda^2g(X,Y)\qquad(X,Y\in\Gamma(\H)).$$
If $\lambda$ is constant, 
$\varphi$ is said to be \emph{horizontally homothetic}. Clearly, any
horizontally conformal submersion to an almost Hermitian manifold is
\emph{PHWC}. 
\begin{corollary}\label{aci}
 Assume that $\varphi:(M,g)\to (N,h,J)$ is a
horizontally homothetic submersion from a Riemannian manifold to an
almost Hermitian manifold. Then
$$\delta\varphi^*\omega=-\lambda^2\musicu\dive f,$$ where $\lambda$ is
the dilation of $\varphi$ and $f$ the associated $f$-structure. 
\end{corollary}
\begin{proof} It is well known that $\nabla\d\varphi(X,Y)=0$ when $X$
and $Y$ are horizontal. Thus, if $Y$ is horizontal, we have
$$\delta\varphi^*\omega(Y)=h(J\varphi_*f\dive
f,\varphi_*Y)=-\lambda^2g(\dive f,Y).$$ On the other hand, if $Y$ is
vertical, then  
\begin{equation*}
\begin{split}
\delta\varphi^*\omega(Y)=&-\sum_ah(\varphi_*fe_a,\nabla\d\varphi(e_a,Y))
=\sum_ah(\varphi_*fe_a,\varphi_*\nabla_{e_a}Y)\\ 
=&-\lambda^2\sum_ag(\nabla_{e_a}fe_a,Y)=-\lambda^2g(\dive f,Y). 
\end{split}
\end{equation*}
\end{proof}

We will use this result to construct a family of homogeneous projections
into Hermitian symmetric spaces, all of which minimize $F$ in their homotopy
class.
Let $\g^\cn$ be a semi-simple complex Lie
algebra, $\g$ be a compact real form of $\g^\cn$ and
$\h\subset\g$ be a Cartan subalgebra with complexification
$\h^\cn\subset\g^\cn$. Denote by $\Delta=\Delta^+\cup\Delta^-$ the set
of roots and its decomposition into positive and negative roots, after
a choice of a positive Weyl chamber. Let $\Pi\subset\Delta^+$ be the
set of simple roots. For any subset $\Pi_0\subset \Pi$, we can
construct a parabolic subalgebra as
$$\p_0=\h^\cn+\sum_{\alpha\in[\Pi_0]}\g_\alpha+
\sum_{\alpha\in\Delta^+\setminus[\Pi_0]}\g_\alpha,$$ where $[\Pi_0]$
is the set of roots in the span of $\Pi_0$.

Let $G^\cn$ be a complex Lie group with Lie algebra $\g^\cn$,
and $G$, $P_0$ and $K_0$ be Lie subgroups of $G^\cn$ with Lie
algebras $\g$, $\p_0$ and $\k_0=\g\cap\p_0$, respectively. Let $\m_0$
be the Killing orthogonal complement to $\k_0$, so that
$$\m_0=\sum_{\alpha\in\Delta^+\setminus[\Pi_0]}(\g_\alpha+\g_{-\alpha})\cap\g.$$
We may identify $\m_0$ with the tangent space of $G/K_0$ at the
identity coset. Now $G/K_0$ has a $G$-invariant integrable complex
structure, for which the $(1,0)$ and $(0,1)$ spaces at the identity
coset are
$$\m_0^{1,0}=\sum_{\alpha\in\Delta^+\setminus[\Pi_0]}\g_\alpha, \quad
\m_0^{0,1}=\sum_{\alpha\in\Delta^+\setminus[\Pi_0]}\g_{-\alpha}.$$ It
is well known that minus the Killing form of $\g$ equips
$G/K_0$ with a Riemannian metric for which this complex structure is
Hermitian and cosymplectic (i.e.\ the K\"ahler form is {\em coclosed}),
see e.g. \cite{Svensson}. 

Now, for any nested pair of subsets $\Pi_0\subset\Pi_0'\subset\Pi$ of simple
roots, we get, with obvious notation, a homogeneous fibration
$$\varphi:G/K_0\to G/K_0',$$ and this map is clearly a holomorphic
Riemannian submersion with totally geodesic fibres, see e.g.,
\cite[page 257]{Besse}.
Hence, in the case where $G/K_0'$ is K\"ahler, $\varphi$ is a critical
point of $F$. 
 From now on, we will assume that $G/K_0'$ is a
Hermitian symmetric space, so that the complex structure is in fact
K\"ahler. 

\begin{proposition} With the notation and conventions introduced above,
the homogeneous fibration $\varphi:G/K_0\ra G/K_0'$ minimizes $F$
in its homotopy class.
\end{proposition}

\begin{proof}
It is, by Theorem \ref{phi}, Corollary \ref{aci} and $G$-invariance, 
enough to show that
$\dive f=0$ at the identity coset. It follows easily from the
formula for the Levi-Civita connection 
in a reductive homogeneous space, see \cite[page 183]{Besse}, that
$$\ip{\nabla_XY}{Z}=0,$$ for all $X,Y\in\g_{\pm\alpha}$ and
$Z\in\g_{\pm\beta}$, for any $\alpha,\beta\in\Delta^+$.
Since $f$ is $G$-invariant and preserves the subspace
$(\g_\alpha+\g_{-\alpha})\cap\g$ of $\m_0$, we have
$$\ip{(\nabla_Xf)X}{Z}=\ip{\nabla_XfX}{Z}-\ip{f\nabla_XX}{Z}=0,$$ for
all $X\in\g_{\pm\alpha}$ and $Z\in\g_{\pm\beta}$, for any
$\alpha,\beta\in\Delta^+$. From the above orthogonal decomposition of
$\m_0$, we see that $\dive f=0$.
\end{proof}

\begin{example} Let $n_1,\dots,n_k$ be positive integers, and
$n=n_1+\dots+n_k$. Define $M$ as the space of
decompositions $$\cn^n=V_1\oplus\dots\oplus V_k,$$ where $\dim
V_i=n_i$, and $V_i\perp V_j$ whenever $i\neq j$. If we denote by
$Gr_{n_1}(\cn^n)$ the Grassmannian of $n_1$-dimensional subspaces of
$\cn^n$, we have an obvious map $$\varphi:M\to Gr_{n_1}(\cn^n),\
\varphi(V_1\oplus\dots\oplus V_k)=V_1.$$ By considering both $M$ and
$Gr_{n_1}(\cn^n)$ as homogeneous $SU(n)$-spaces, it follows easily
that $\varphi$ is a homogenous projection of the type considered 
above, and hence a minimizer of $F$ in its homotopy class.  
\end{example}

\begin{remark} Let $\varphi:M\ra N$ be a holomorphic map between almost
Hermitian manifolds. It is well known that if $M$ and 
$N$ are almost K\"ahler then $\varphi$ is harmonic and minimizes 
the Dirichlet energy $E$ in its
homotopy class, see e.g., \cite{Eells-Lemaire}
for a proof. In fact, this proof
immediately generalizes to the case where $M$ is only cosymplectic.
Hence, if, in addition, $\varphi^*\omega$ is coclosed, we see that
$\varphi$ minimizes the full Faddeev-Hopf energy $E+\alpha F$ for all
$\alpha>0$. This applies to the homogeneous fibrations
$G/K_0\ra G/K_0'$ defined above.
\end{remark}

\end{document}